\documentclass{amsart}

\newcommand{\R}{\mathbb R}

\newcommand{\co}{\colon}
\renewcommand{\phi}{\varphi}
\newcommand\ka{\kappa}
\newcommand\ga{\gamma}
\newcommand\ep{\varepsilon}
\newcommand\pd{\partial}
\def\<{\langle}
\def\>{\rangle}

\newcommand\Gr{\operatorname{Gr}}
\newcommand\id{\operatorname{id}}
\newcommand\Lin{\operatorname{Lin}}

\newcommand\U{\mathcal U}

\newcommand{\be}{\begin{equation}}
\newcommand{\ee}{\end{equation}}

\newtheorem{theorem}{Theorem}

\newtheorem{lemma}{Lemma}[section]
\newtheorem{prop}[lemma]{Proposition}
\newtheorem{cor}[lemma]{Corollary}
\newtheorem*{fact}{Fact}
\newtheorem*{tsecprime}{Theorem \ref{t:sections}'}

\theoremstyle{definition}
\newtheorem{definition}[lemma]{Definition}
\newtheorem{question}[lemma]{Question}
\newtheorem{example}[lemma]{Example}

\theoremstyle{remark}
\newtheorem{remark}[lemma]{Remark}

\numberwithin{equation}{section}

\begin{document}

\title[Monochromatic Finsler surfaces]{Monochromatic Finsler surfaces and a local ellipsoid characterization}

\author{Sergei Ivanov}
\address{St.~Petersburg Department of Steklov Mathematical Institute of
Russian Academy of Sciences,
Fontanka 27, St.Petersburg 191023, Russia}
\address{Saint Petersburg State University, 
7/9 Universitetskaya emb., St. Petersburg 199034, Russia
}
\email{svivanov@pdmi.ras.ru}



\thanks{Research is supported by the Russian Science Foundation grant 16-11-10039}

\begin{abstract}
We prove the following localized version of a classical ellipsoid characterization:
Let $B\subset\R^3$ be convex body with a smooth strictly convex boundary and 
0 in the interior, and suppose that there is an open set of planes
through 0 such that all sections of $B$ by these planes are linearly equivalent.
Then all these sections are ellipses and the corresponding part of $B$
is a part of an ellipsoid.

We apply this to differential geometry of Finsler surfaces
in normed spaces and show that
in certain cases the intrinsic metric of a surface
imposes restrictions on its extrinsic geometry
similar to implications of Gauss' Theorema Egregium.
As a corollary we construct 2-dimensional Finsler metrics
that do not admit local isometric embeddings
to dimension~3.
\end{abstract}

\maketitle

\section{Introduction}

Although this paper is motivated by Finsler geometry,
one of the main results is the following
theorem about convex sets.

\begin{theorem}\label{t:sections}
Let $B\subset\R^n$, $n\ge 3$, be a convex body with a smooth strictly convex boundary
and 0 in the interior.
Let $H\subset\R^n$ be a two-dimensional plane through~0
and suppose that there is a neighborhood of $H$
in the Grassmannian $\Gr_2(\R^n)$
such that for every plane $H'$ from this neighborhood
the cross-section $H'\cap B$ is linearly equivalent to $H\cap B$.

Then $H\cap B$ is an ellipse centered at~0
and furthermore $B$ coincides with an ellipsoid
in a neighborhood of $H$.
\end{theorem}

Here by \textit{linear equivalence} of two cross-sections
we mean the existence of a linear bijection between their planes
sending one cross-section to the other.
%(this relation could also be called pointed affine equivalence).
The word ``smooth'' in this paper always means $C^\infty$,
and strict convexity means quadratic strict convexity,
i.e.\ positivity of appropriate second derivatives.
By $\Gr_k(V)$ we denote
the Grassmannian of $k$-dimensional linear subspaces of
a vector space~$V$.

Theorem \ref{t:sections} is a ``local version''
of the following classical ellipsoid characterization
due to Auerbach, Mazur and Ulam~\cite{AMU}:
{\it If all planar cross-sections through 0 of a convex body
are linearly equivalent, then the body is a centered ellipsoid.}
(Note that in this contexts local statements are stronger than global ones.)
%(for smooth strictly convex bodies).

\begin{remark}
The following more general question goes back to Banach \cite{Banach32} and
remains only partially solved: For given $k$ and $n$ ($k<n$),
are Euclidean spaces the only $n$-dimensional Banach spaces
with the property that all $k$-dimensional linear subspaces are isometric?
The above mentioned Auerbach-Mazur-Ulam's theorem
answers this question for $k=2$.
Gromov \cite{Gro67} proved that the answer is affirmative
if $k$ is even or $n\ge k+2$.
The case when $k$ is odd and $n=k+1$ remains open.
See also \cite[Note~7.2]{Gardner} for discussion and related results.

The proofs in \cite{AMU} and \cite{Gro67} are based on global
obstructions arising from algebraic topology of Grassmannians.
These methods do not work in the local version of the problem.
\end{remark}

\begin{remark}
It is plausible that the smoothness assumption in 
Theorem \ref{t:sections} can be relaxed.
However (some form of) strict convexity is necessary.
Indeed, pick any convex body $B_0\subset\R^2$  and let $B\subset\R^2\times\R$ 
be a convex body which coincides with $B_0\times\R$
in a neighborhood of the plane $H=\R^2\times\{0\}$.
Then all cross-sections of $B$ by planes sufficiently close to~$H$ are linearly equivalent to~$B_0$
but $B_0$ is not necessarily an ellipse.
\end{remark}

Every convex body with 0 in the interior is the unit ball of 
a norm.
We use the word ``norm'' in a slightly generalized meaning,
namely a norm is not required to be symmetric.
Smooth strictly convex bodies correspond to especially nice norms
called Banach-Minkowski ones.
By definition, a \textit{Banach-Minkowski norm} on a vector space $V$ 
is a (possibly non-symmetric) norm $\Phi\co V\to\R_+$
%positively 1-homogeneous, positive and 
which is smooth outside 0
and such that the function $\Phi^2$ is (quadratically) strictly convex.
A \textit{Banach-Minkowski space} is a finite-dimensional
vector space equipped with a Banach-Minkowski norm.
These spaces are often referred to as Minkowski spaces but
we use the longer term to avoid confusion with special relativity.

A norm is called \textit{Euclidean} if it is
associated with an inner product.
The following is a slightly more detailed reformulation 
of Theorem \ref{t:sections} in terms of norms.

\begin{tsecprime}
%\begin{prop}\label{p:sections}
Let $V=(V^n,\Phi)$ be a Banach-Minkowski space, $n\ge 3$, and
assume that $\U\subset\Gr_2(V)$ is a connected open set such that
for all $H,H'\in\U$ the normed planes $(H,\Phi|_H)$ and $(H',\Phi|_{H'})$ are isometric.

Then $\Phi|_H$ is a Euclidean norm for every $H\in\U$.
Moreover there exists a Euclidean norm on $V$ whose restriction
to every plane from $\U$ coincides with the restriction of $\Phi$.
%\end{prop}
\end{tsecprime}

\begin{remark}
The assumption that $\U$ is connected
is not necessary if ${n=3}$.
In higher dimensions one can construct examples where $\Phi$ coincides
with one Euclidean norm near one plane and
with another Euclidean norm near some other plane.
The details of this construction are left to the reader.
\end{remark}

If $B$ is the unit ball of a norm $\Phi$
and two planar cross-sections $H\cap B$ and $H'\cap B$ are linearly
equivalent, then the normed planes $(H,\Phi|_H)$ and $(H',\Phi|_{H'})$ are isometric
(and vice versa).
Thus Theorem \ref{t:sections}' implies Theorem \ref{t:sections}.

Theorem \ref{t:sections}' easily follows from its 3-dimensional case
(see Section \ref{sec:riem}).
Finslerian results discussed below work only in 3-dimensional spaces.
For these reasons in what follows we mainly restrict our attention to dimension~3.

\subsection*{Finsler surfaces}
For a smooth surface $M^2\subset\R^3$ and a point $p\in M$,
the celebrated Gauss' Theorema Egregium implies that
local intrinsic geometry of $M$ near~$p$
determines the type of the second-order extrinsic geometry at $p$.
Namely the sign of the Riemannian curvature of $M$ at $p$ determines
whether the second fundamental form of $M$ at~$p$
is definite, semi-definite, or degenerate.
(We call a quadratic form \textit{definite} if it either positive definite
or negative definite.)

The distinction between the three types of the second fundamental form is affine invariant.
Hence one may ask whether there are similar relations
between intrinsic and extrinsic geometry for Finsler surfaces in normed 3-spaces.
More precisely, we have the following question.

\begin{question}[{cf.~\cite[Remark 1.6]{BI11}}]\label{main-question}
Let $M$ be a two-dimensional Finsler manifold and
$f_i\co M\to V_i$, $i=1,2$, smooth isometric embeddings
to 3-dimensional Banach-Minkowski spaces.
For $p\in M$, is it always true that the 
second fundamental forms of $f_i$ at $p$
are of the same type: definite, semi-definite, or degenerate?
\end{question}

\begin{remark}\label{rem:highdim}
If the dimension of the ambient space $V$ is greater than 3, then the answer to any sensible variant
of Question \ref{main-question} is negative. 
Indeed, as shown in \cite{BI11}, 
locally any 2-dimensional Finsler manifold is isometric 
to a strictly saddle surface in $\R^4$ equipped with some Banach-Minkowski norm.
%In particular, a saddle surface in a normed 4-space can be isometric
%to a small spherical cap with its standard ``round'' Riemannian metric.
In contrast, in the Euclidean case this is possible
only for negatively curved metrics.
\end{remark}

In this paper we show that at least for some Finsler metrics
the answer to Question \ref{main-question} is affirmative.
Namely this is the case if the metric is second-order flat
at the point in question or monochromatic
(see below for definitions).
In both cases the second fundamental form %of an isometric embedding
is necessarily degenerate, unless in the monochromatic case
the metric is actually Riemannian.

Recall that a \textit{Finsler manifold} is a smooth manifold $M$ equipped with 
a Finsler metric, that is a continuous function $\phi\co TM\to\R_+$ such that
$\phi$ is smooth on $TM\setminus 0$ and $\phi|_{T_xM}$
is a Banach-Minkowski norm for every $x\in M$.
Every Banach-Minkowski space is naturally a Finsler manifold.
An isometric immersion is a norm-preserving immersion
from one Finsler manifold to another.
Since our set-up is local, we consider
isometric embeddings rather than immersions.

The standard definition of the second fundamental form %of a surface
requires an inner product, which we do not have in our set-up.
We use the following affine invariant version of the definition.
Let $M$ be a smooth manifold, $V$ a vector space,
and ${f\co M\to V}$ a smooth immersion.
Then the second fundamental form of $f$ at $p\in M$ 
is a symmetric bilinear form on $T_pM$ with
values in the quotient vector space $V/\operatorname{Im}d_pf$
where $d_pf\co T_pM\to V$ is the differential of $f$ at~$p$.
This form is defined as the Hessian of $\pi\circ f$ at $p$
where $\pi\co V\to V/\operatorname{Im}d_pf$ is the quotient projection.
The Hessian is well-defined since $p$ is a critical point of $\pi\circ f$.
If the quotient space is one-dimensional, one may regard the second fundamental form
as a real-valued form defined up to a multiplication by a constant.

Following S.-S.~Chern and D.~Bao (see \cite[\S3.3]{Bao}),
we call a Finsler manifold $M=(M,\phi)$ \textit{monochromatic}
if all its tangent spaces are isometric as normed vector spaces.
%That is, for every $x,y\in M$ there exists a linear bijection
%$I\co T_xM\to T_yM$ such that $\phi(I(v))=\phi(v)$ for all $v\in T_xM$.
Note that all Riemannian manifolds are monochromatic.

\begin{theorem}\label{t:mono}
Let $M^2$ be a monochromatic non-Riemannian Finsler manifold,
$V^3$ a Banach-Minkowski space and $f\co M\to V$ a smooth isometric embedding.
Then the second fundamental form of $f$ is degenerate at every point.
\end{theorem}

In a sense, Theorem \ref{t:mono} is a reformulation of Theorem~\ref{t:sections},
see Remark \ref{rem:converse}.

The property that the second fundamental form is degenerate at every point
is a strong restriction.
We use implications of this property and Theorem \ref{t:mono}
to construct examples of Finsler metrics in $\R^2$
that do not admit local isometric embeddings to 3-dimensional
Banach-Minkowski spaces. See Example \ref{ex:nonembed} and Proposition \ref{p:nonembed}.

As already mentioned in Remark \ref{rem:highdim},
all Finsler surfaces admits local isometric embeddings to dimension~4.
Thus $n=4$
is the minimal $n$ such that all Finsler surfaces admit
local isometric embeddings to dimension~$n$.
Note that there is no such universal dimension for \emph{global} isometric
embeddings and moreover non-compact Finsler manifolds generally
do not admit isometric embeddings to Banach-Minkowski spaces,
see \cite{BI93} and~\cite{Shen}.

\medskip

%Now we formulate our result on second-order flat metrics.

\begin{definition}\label{d:soflat}
Recall that a Finsler metric is called \textit{flat}
if it is locally isometric to a Banach-Minkowski space.

Let $M=(M,\phi)$ be a Finsler manifold.
We say that the metric $\phi$ is \textit{second-order flat}
at a point $p\in M$ if there exists a flat Finsler metric $\phi_0$
defined in a neighborhood of $p$ such that for $x\in M$ near $p$ and 
$v\in T_xM\setminus\{0\}$ one has
\be\label{e:2oflat}
  \frac{\phi(x,v)}{\phi_0(x,v)} = 1 + o(|x-p|^2), \quad x\to p,
\ee
where $|x-p|$ is the distance from $x$ to $p$ 
in an arbitrary local coordinate system.
\end{definition}

If $\phi$ is Riemannian, i.e.\ $\phi=\sqrt g$ where $g$ is a Riemannian metric tensor,
then \eqref{e:2oflat} is equivalent to the existence of local coordinates
in which the second derivatives of~$g$ vanish at~$p$.
In dimension~2 this is equivalent to $K(p)=0$
where $K$ is the curvature of the metric.
By Gauss' Theorem this implies
that every smooth isometric embedding of $M$ to $\R^3$
has a degenerate second fundamental form at~$p$.
In the next theorem we generalize this implication to Finsler surfaces.

\begin{theorem}\label{t:flat}
Let $M=(M^2,\phi)$ be a Finsler manifold whose metric is second-order flat
at a point $p\in M$. % (see Definition \ref{d:soflat}).
Let $V^3$ be a Banach-Minkowski space and $f\co M\to V$ a smooth isometric embedding.
Then the second fundamental form of $f$ at $p$ is degenerate.
\end{theorem}

Theorems \ref{t:mono} and \ref{t:flat} %and Gauss' theorem for Riemannian metrics
imply that the answer to Question \ref{main-question} is affirmative if
the Finsler metric is second-order flat at $p$
or monochromatic (Riemannian or not), see Corollary \ref{cor:mono}.

The rest of the paper is organized as follows.
We begin with the proof of Theorem~\ref{t:flat} in Section~\ref{sec:flat}.
Subsequent arguments do not depend on Theorem~\ref{t:flat} directly
but they use similar ideas.
In Section~\ref{sec:mono} we prove the first part of
Theorem~\ref{t:sections}' (see Proposition \ref{p:sections})
and deduce Theorem~\ref{t:mono}.
In Section~\ref{sec:riem} we finish the proof of Theorem \ref{t:sections}'
and obtain some corollaries.
%consider isometric embeddings of Riemannian 2-manifolds to
%Banach-Minkowski spaces.
Finally in Section \ref{sec:example} we
construct examples of metrics that are not locally embeddable to dimension~3.

\subsubsection*{Acknowledgement}
The author thanks Vladimir Matveev and Fedor Petrov for
inspiring discussions and help in finding references,
and Dmitri Burago for his comments on the
first draft of the paper.

\section{Proof of Theorem \ref{t:flat}}
\label{sec:flat}

%In this section we prove Theorem \ref{t:flat}.
Let $(M,\phi)$ and $p$ be as in Theorem \ref{t:flat}
and $\phi_0$ as in Definition~\ref{d:soflat}.
Since the statement of the theorem is local,
we may restrict ourselves to a small coordinate neighborhood $U$ of~$p$.
Thus we may assume that $M=U\subset\R^2$
and use the standard identification $TU=U\times\R^2$.
Then $\phi$ is a function of two variables $x\in U$ and $v\in\R^2$.

We further assume that our local coordinates
are chosen so that $\phi_0$ is the restriction of 
a Banach-Minkowski norm. In other words, $\phi_0(x,v)=\phi_0(v)$
does not depend on~$x$.
Then \eqref{e:2oflat} implies that for every $v\in\R^2$
we have $\phi(p,v)=\phi_0(v)$,
\be
\label{e:d1phi=0}
 d_p\big(\phi(\cdot,v)\big) = 0,
\ee
and
\be
\label{e:d2phi=0}
 d^2_p\big(\phi(\cdot,v)\big) = 0.
\ee
Here and below $d_p$ denotes the differential of a function at $p$,
$d^2_p$ the second differential at $p$, etc.

Let $V=(V,\Phi)$ be a 3-dimensional Banach-Minkowski space and
$f\co(U,\phi)\to V$ a smooth isometric embedding. Then
\be
\label{e:iso}
 \Phi(d_xf(v)) = \phi(x,v)
\ee
for all $x\in U$ and $v\in\R^2$.
Define $H=\operatorname{Im} d_pf$, that is $H$ is the tangent plane to $f(U)$ at~$f(p)$,
regarded as a linear subspace of~$V$.
The map $d_pf$ is a linear isometry between $(\R^2,\phi_0)$ and $(H,\Phi|_H)$.
Fix an isomorphism between $V/H$ and $\R$ and denote by $S$ the second fundamental
form of $f$ at $p$ composed with this isomorphism. That is, $S$ is a symmetric 
real-valued bilinear form on $\R^2$ given by
$S(v,w)=\pi (d^2_pf(v,w))$ for all $v,w\in\R^2$,
where $\pi\co V\to V/H\cong\R$ is the quotient map.
Our goal is to prove that $S$ is degenerate.
We say that vectors $v,w\in\R^2$ are $S$-orthogonal if $S(v,w)=0$.

Let $v\in\R^2$ and $\tilde v=d_pf(v)$.
Differentiating \eqref{e:iso} with respect to $x$ at $x=p$
in the direction $w\in\R^2$ and taking into account \eqref{e:d1phi=0}
yields
\be\label{e:d1iso}
d_{\tilde v}\Phi (d^2_pf(v,w)) = 0.
\ee
Here $d_{\tilde v}\Phi$ denotes the differential of $\Phi$ at $\tilde v$,
this differential is a linear map from $V$ to $\R$,
and $d^2_pf(v,w)\in V$ is an argument of this linear map.
Subsequent formulas involving derivatives should be read in a similar way.

Differentiating \eqref{e:iso} twice with respect to $x$ at $x=p$
in directions $w,w_1\in\R^2$ and taking into account \eqref{e:d2phi=0}
yields
\be\label{e:d2iso}
 d_{\tilde v}\Phi (d^3_pf(v,w,w_1)) + d^2_{\tilde v}\Phi(d^2_pf(v,w),d^2_pf(v,w_1)) = 0.
\ee
We fix the notation $\tilde v=d_pf(v)$ for the rest of the proof.

Let $\Sigma$, $\gamma$ and $\tilde\gamma$ denote the unit spheres
of $\Phi$, $\phi_0$ and $\Phi|_H$, respectively.
That is,
$$
\begin{aligned}
\Sigma &= \{v\in V:\Phi(v)=1\}, \\
\gamma &= \{v\in\R^2:\phi_0(v)=1\}, \\
\tilde\gamma &= \Sigma\cap H=d_pf(\gamma).
\end{aligned}
$$
Note that $\gamma$ and $\tilde\gamma$ are
smooth strictly convex curves enclosing the origin in their respective planes.
For $v\in\gamma$, we denote
by $\ell_v\subset\R^2$ the tangent direction of $\gamma$ at $v$
(that is the one-dimensional linear subspace of~$\R^2$ parallel
to the tangent to $\gamma$ at~$v$)
and define $\tilde\ell_v=d_pf(\ell_v)$.

\begin{lemma}\label{l:tau}
There exists a vector $\tau\in V\setminus H$ such that
\be\label{e:tau}
 d^2_pf(v,w) = S(v,w) \cdot \tau
\ee
for all $v,w\in\R^2$.
\end{lemma}

\begin{proof}

Let $v,w\in\gamma$ be linearly independent $S$-orthogonal vectors.
We are going to show that $d^2_p(v,w)=0$.

First assume that $\ell_v\ne\ell_w$.
Let $\xi=d^2_pf(v,w)$.
Since $S(v,w)=0$, we have $\xi\in H$ by the definition of the second fundamental form.
By \eqref{e:d1iso} we have
$ d_{\tilde v}\Phi(\xi) = 0 $ where $\tilde v=d_pf(v)$.
This means that $\xi$ is a tangent vector to $\Sigma$ at~$\tilde v$.
Since $\xi\in H$, it follows that $\xi\in\tilde\ell_v$.
Interchanging $v$ and $w$ yields that $\xi\in\tilde\ell_w$,
hence $\xi\in\tilde\ell_v\cap\tilde\ell_w$.
Since $\tilde\ell_v\ne\tilde\ell_w$,
it follows that $\xi=0$ as claimed.

Now consider the case $\ell_v=\ell_w$.
In this case one can apply the above argument to the opposite
unit vector $v'=\frac{-v}{\phi_0(-v)}$ in place of $v$.
Indeed, $\ell_{v'}\ne\ell_w$
since otherwise $\ell_v=\ell_w=\ell_{v'}$,
contrary to the strict convexity of $\gamma$.
Therefore $d^2_pf(v',w)=0$ and hence $d^2_pf(v,w)=0$.

Thus $d^2_pf(v,w)=0$ for all pairs of vectors $v,w\in\R^2$ 
that are linearly independent and $S$-orthogonal.
To finish the proof we need the following fact,
which is left as an exercise to the reader.

\begin{fact}
Let $S$ and $S_1$ be symmetric bilinear forms on $\R^2$ such that
$S_1(v,w)=0$ for all pairs of linearly independent $S$-orthogonal vectors $v,w\in\R^2$.
Then $S_1=\lambda S$ for some $\lambda\in\R$. 
\end{fact}

Pick a basis $e_1,e_2,e_3$ of $V$ such that $e_1,e_2\in H$ and 
$e_3$ is mapped to $1\in\R$ by
the quotient map $\pi\co V\to V/H\cong\R$
(recall that we have fixed an isomorphism between $V/H$ and $\R$).
Decompose $d^2_pf$ into coordinate components with respect to this basis:
$d^2_pf = S_1e_1+S_2e_2+S_3e_3$.
Applying the above fact to $S_1$ and $S_2$ yields that
$S_1=\lambda_1 S$ and $S_2=\lambda_2 S_2$ for some $\lambda_1,\lambda_2\in\R$.
And $S_3=S$ by the choice of~$e_3$.
Thus \eqref{e:tau} holds for $\tau=\lambda_1e_1+\lambda_2e_2+e_3$.
\end{proof}

From now on we assume that $S\ne 0$
(otherwise the statement of Theorem \ref{t:flat} is trivial).
Let $\tau$ be the vector from Lemma \ref{l:tau}
and $P\co V\to H$ the projector such that $P(\tau)=0$.
Substituting \eqref{e:tau} into \eqref{e:d1iso} yields
$$
 S(v,w)\cdot d_{\tilde v}\Phi(\tau) = 0.
$$
Since $S\ne 0$, it follows that
\be\label{e:dPhi(tau)}
 d_{\tilde v}\Phi(\tau) = 0
\ee
for almost all $\tilde v\in\tilde\gamma$ and hence for all $\tilde v\in\tilde\gamma$.
%(Geometrically this means that one vector $\tau$ is tangent 
%to $\Sigma$ at all point of $\tilde\gamma=\Sigma\cap H$.
%Equivalently, the operator norm of the projector $P$ equals~1.)

For $v\in\gamma$ define
$ \ka(v) = d^2_{\tilde v}\Phi(\tau,\tau) $.
Due to strict convexity of $\Phi$ we have $\ka(v)>0$ for all $v\in\ga$.
Define a map $T\co\R^2\times\R^2\times\R^2\to H$ by $T=P\circ d^3_pf$.
This map is symmetric and linear in each argument.
Due to \eqref{e:dPhi(tau)}, the first term in \eqref{e:d2iso} equals
$$
d_{\tilde v}\Phi (d^3_pf(v,w,w_1)) = d_{\tilde v}\Phi (T(v,w,w_1)) .
$$
The second term in \eqref{e:d2iso} equals
$$
 d^2_{\tilde v}\Phi(d^2_pf(v,w),d^2_pf(v,w_1)) = \ka(v) S(v,w) S(v,w_1)
$$
by \eqref{e:tau} and the definition of $\ka(v)$.
These identities and \eqref{e:tau} imply that
\be
\label{e:dPhiT}
 d_{\tilde v}\Phi (T(v,w,w_1)) + \ka(v) S(v,w) S(v,w_1) = 0
\ee
for all $v,w,w_1\in\R^2$.

\begin{lemma}\label{l:T=0}
If $v,w\in\ga$ are linearly independent and $S$-orthogonal,
then 
$$
T(v,w,w_1)=0
$$
for all $w_1\in\R^2$.
\end{lemma}

\begin{proof}
Similarly to the proof of Lemma \ref{l:tau}
we may assume that $\ell_v\ne\ell_w$,
otherwise replace $v$ by $v'=\frac{-v}{\phi_0(-v)}$.
Since $S(v,w)=0$, the second term in \eqref{e:dPhiT} vanishes,
hence $d_{\tilde v}\Phi (T(v,w,w_1))=0$.
Since $T(v,w,w_1)\in H$, it follows that
$T(v,w,w_1)\in\tilde\ell_v$.
Interchanging $v$ and $w$ yields that $T(v,w,w_1)=T(w,v,w_1)\in\tilde\ell_w$.
Since $\tilde\ell_v\ne\tilde\ell_w$, it follows that
$T(v,w,w_1)\in\tilde\ell_v\cap\tilde\ell_w=\{0\}$.
\end{proof}

Fix $v$ and $w$ as in Lemma \ref{l:T=0}
and let $w_1$ range over $\R^2$.
By Lemma \ref{l:T=0} the first term in \eqref{e:dPhiT} vanishes.
Since $\ka(v)\ne 0$, it follows that 
$$
S(v,w_1)S(w,w_1)=0
$$
for all $w_1\in\R^2$.
This implies that one of the linear functions $S(v,\cdot)$ and $S(w,\cdot)$
vanishes everywhere on~$\R^2$. Equivalently,
either $v$ or $w$ belong to the kernel of~$S$.
Therefore $S$ is degenerate.
This finishes the proof of Theorem~\ref{t:flat}.

\section{Linearly equivalent sections}
\label{sec:mono}

In this section we consider 3-dimensional Banach-Minkowski spaces
with many linearly equivalent cross-sections.
Our main goal is to prove the first claim of Theorem \ref{t:sections}' in dimension~3.
We restate it as the following proposition.

\begin{prop}\label{p:sections}
Let $V=(V^3,\Phi)$ be a Banach-Minkowski space and
$\U\subset\Gr_2(V)$ an open set such that
all normed planes $(H,\Phi|_H)$, $H\in\U$, are isometric.
Then $\Phi|_H$ is a Euclidean norm for every $H\in\U$.
\end{prop}

Theorem \ref{t:sections} follows from Proposition \ref{p:sections}
and Lemma \ref{l:quadratic}, see the next section.
In this section we prove Proposition \ref{p:sections}
and deduce Theorem \ref{t:mono} from it.
The proof of Proposition \ref{p:sections} is based on differential-geometric
analysis of $\Phi$ near a plane $H\in\U$.
The key result of this analysis is Lemma \ref{l:Rtau}.

\medbreak

Let $(V,\Phi)$ and $\U$ be as in Proposition \ref{p:sections}.
Fix $H\in\U$ and a vector $\nu\in V\setminus H$.
First we define a convenient local parametrization of $\Gr_2(V)$.
Namely a neighborhood of $H$ in $\Gr_2(V)$ is parametrized by
a neighborhood of 0 in the dual space $H^*$ as follows.
For $\alpha\in H^*$ let $\pi_\alpha\co V\to\R$ be the unique
linear function such that $\pi_\alpha|_H=\alpha$ and $\pi_\alpha(\nu)=1$,
and let $H_\alpha=\ker \pi_\alpha$.
Then $H_0=H$ and the map $\alpha\mapsto H_\alpha$ is a diffeomorphism
from $H^*$ onto the set of all planes from $\Gr_2(V)$ that do not contain $\nu$.
We restrict this diffeomorphism to a neighborhood $U$ of 0 in $H^*$ 
such that $H_\alpha\in\U$ for all $\alpha\in U$.
Throughout the proof we keep replacing this neighborhood by smaller ones
but use the same notation $U$ for all these neighborhoods.

\begin{remark}\label{r:coord}
Let $x,y,z$ be coordinates in $V$ associated to a basis $e_1,e_2,e_3$
where $e_1,e_2\in H$ and $e_3=\nu$.
Then the map $\alpha\mapsto H_\alpha$ can be described as follows:
$H_\alpha$ is the plane defined by the equation $ax+by+z=0$
where $(a,b)$ are the coordinates of $\alpha\in H^*$ 
with respect to the dual basis.
\end{remark}

The map $\alpha\mapsto\pi_\alpha$ is affine.
Its linear component is the map $I\co H^*\to V^*$
defined by the relations
$I(\alpha)|_H = \alpha$ and $I(\alpha)(\nu) = 0$ 
%\be\label{e:I}
% \pi_\alpha = \pi_0 + I(\alpha), 
% \qquad 
% I(\alpha)|_H = \alpha,
% \qquad
% I(\alpha)(\nu) = 0
%\ee
for all $\alpha\in H^*$.
We will need the following calculation:
If $t\mapsto\alpha(t)$ is a smooth path in $H^*$ with $\alpha(0)=0$
and $\dot\alpha(0)=h$,
and $t\mapsto v(t)$ is a smooth path in $V$ with $v(0)=v_0\in H$, then
\be\label{e:dpialpha}
 \tfrac d{dt}\big|_{t=0} \pi_{\alpha(t)}(v(t))
 = \pi_0(\dot v(0)) + I(\dot\alpha(0))(v(0)) = \pi_0(\dot v(0)) + h(v_0) .
\ee

Let $\phi_\alpha$ denote the restriction of $\Phi$ to $H_\alpha$.
Similarly to the previous section we denote
by $\Sigma$ the unit sphere of $\Phi$,
by $\ga$ the unit circle of $(H,\phi_0)$,
and by $\ell_v$ the tangent direction to~$\ga$ at $v\in\ga$.
We extend the notation $\ell_v$ to all $v\in H\setminus\{0\}$
by homogeneity: $\ell_{\lambda v}=\ell_v$ for all $\lambda>0$.
For vector spaces $X$ and $Y$ we denote 
by $\Lin(X,Y)$ the space of all linear maps from $X$ to~$Y$.

Our assumptions 
imply that $(H_\alpha,\phi_\alpha)$ is isometric to $(H,\phi_0)$
for all $\alpha\in U$.
Arguing by contradiction, suppose that the norm $\phi_0$ is not Euclidean.
We need the following technical lemma.
%Then the group of self-isometries of $(H,\phi_0)$ is finite.
%It follows that if $U$ is sufficiently small then
%there is a unique continuous family $\{L_\alpha\}_{\alpha\in U}$ of linear isometries
%$L_\alpha\co (H,\phi_0)\to (H_\alpha,\phi_\alpha)$ such that $L_0=\id_H$.
%The smoothness of $\Phi$ implies that this family is smooth.

\begin{lemma}\label{l:smooth}
If $U$ is sufficiently small then there exists a smooth family
$\{L_\alpha\}_{\alpha\in U}$ of linear maps from $H$ to $V$
such that $L_0=\id_H$, $L_\alpha(H)=H_\alpha$, and
%, $\Phi\circ L_\alpha=\phi_0$ for each $\alpha$,
$L_\alpha$ is an isometry from $(H,\phi_0)$ to $(H_\alpha,\phi_\alpha)$
for each $\alpha\in U$.
\end{lemma}

\begin{proof}
Since $\phi_0$ is not Euclidean, the group of self-isometries of $(H,\phi_0)$ is discrete.
It follows that (possibly for a smaller neighborhood $U$)
there is a unique continuous family $\{L_\alpha\}_{\alpha\in U}$
of isometries $L_\alpha\co (H,\phi_0)\to (H_\alpha,\phi_\alpha)$ such that $L_0=\id_H$.
We have to prove that this family is smooth.

It suffices to show that the set
$\{L_\alpha\}$ is a 2-dimensional smooth submanifold of $\Lin(H,V)$
and this submanifold is transverse to $\Lin(H,H)$.
Let us show that the maps $L_\alpha$ are solutions
of four equations $\{f_i=0\}_{i=1}^4$
where the functions $f_i\co\Lin(H,V)\to\R$ are
smooth in a neighborhood of $L_0$ and satisfy
\be\label{e:regular}
\textstyle \bigcap_{i=1}^4\ker (d_{L_0}f_i) \cap \Lin(H,H) = \{0\} .
\ee
The desired fact follows from the existence of such $f_i$'s
and the implicit function theorem.

We define the functions $f_i$ by 
\be\label{e:fiL}
f_i(L)=\Phi(L(v_i))-1, \qquad L\in\Lin(H,V), 
\ee
for suitably chosen vectors $v_1,v_2,v_3,v_4\in\ga$.
Such functions satisfy $f_i(L_\alpha)=0$ since 
$L_\alpha$ are isometries.
They are smooth in a neighborhood of $L_0$ since $\Phi$
is smooth outside~0. It remains to prove that one can choose $v_1,v_2,v_3,v_4\in H$
so that the maps $f_i$ defined by \eqref{e:fiL} satisfy~\eqref{e:regular}.

Suppose the contrary. Fix $v_1,v_2,v_3\in\ga$ such that they are pairwise
linearly independent and the lines $\ell_{v_i}$, $i=1,2,3$, are pairwise different.
Let $v_4$ range over~$\ga$. 
If $A\in\Lin(H,H)$ belongs to the kernel of $d_{L_0}f_i$,
then $d_{v_i}\phi_0(A(v_i))=0$
and hence $A(v_i)\in\ell_{v_i}$.
Observe that the conditions $A(v_i)\in\ell_{v_i}$ for $i=1,2,3$
determine a linear map $A$ uniquely up to a scalar factor.
Hence there exists a nonzero $A\in\Lin(H,H)$
such that $d_v\phi_0(A(v))=0$ for all $v=v_4\in\ga$ and hence for all $v\in H$.
This implies that $\phi_0$ is preserved by the flow
generated by the linear vector field $v\mapsto A(v)$ on $H$.
Thus $(H,\phi_0)$ admits a one-parameter group of self-isometries,
a contradiction.
\end{proof}

Let $\{L_\alpha\}$ be the family from Lemma \ref{l:smooth}.
Let $R$ denote the differential of this family at~0,
that is $R\co H^*\to \Lin(H,V)$ is a linear map 
given by $R = d_0(\alpha\mapsto L_\alpha)$.

Let $\alpha\in U$, $h\in H^*$ and $v\in H$.
Since $L_\alpha(H)=H_\alpha$, we have $\pi_\alpha(L_\alpha(v))=0$.
Differentiating this identity with respect to $\alpha$ at $\alpha=0$
in the direction $h$ 
and taking into account \eqref{e:dpialpha} we obtain
\be\label{e:diffLa}
 \pi_0(R(h)(v)) + h(v) = 0 .
\ee
Since $L_\alpha$ is an isometry, we have
$\Phi(L_\alpha(v)) = \phi_0(v)$.
Differentiating this identity
with respect to $\alpha$ at $\alpha=0$
in the direction $h$ yields
\be\label{e:dPhiRhv}
 d_v\Phi(R(h)(v)) = 0
\ee
provided that $v\ne 0$.

\begin{lemma}\label{l:Rhv}
If $h\in H^*$ and $v\in\ker h$, then $R(h)(v)=0$.
\end{lemma}

\begin{proof}
If $v\in\ker h$ then the second term in \eqref{e:diffLa} vanishes
and therefore
$$
 R(h)(v) \in \ker\pi_0 = H .
$$
This and \eqref{e:dPhiRhv} imply that 
\be\label{e:Rhv-tangent}
 R(h)(v)\in\ell_v \quad\text{for all $h\in H^*$ and $v\in\ker h\setminus\{0\}$}.
\ee

Pick a linear isomorphism $J\co H\to H^*$ such that
$v\in\ker J(v)$ for all $v\in H$.
(An example of such $J$ is given by $J(x,y)=(-y,x)$
in coordinates from Remark \ref{r:coord}.)
Define 
$$
Q(v)=R(J(v))(v) .
$$
The map $v\mapsto Q(v)$ is an $H$-valued quadratic form on~$H$.
We regard $Q$ as a vector field on $H$.
By \eqref{e:Rhv-tangent}, $Q$ is tangent to level sets of $\phi_0$.

We are to show that $Q=0$.
First observe that $Q$ vanishes at some point $v_0\in\ga$
(and hence on the entire line generated by $v_0$).
Indeed, if this is not the case then $Q$ has a constant orientation
along $\ga$ (i.e., it is directed either
``clockwise'' or ``counter-clockwise'' everywhere).
On the other hand, for opposite points
$v\in\ga$ and $v'=\frac{-v}{\phi_0(-v)}$ 
the vectors $Q(v)$ and $Q(v')$ are positively proportional
due to quadratic homogeneity of $Q$,
hence they have opposite orientations on~$\ga$, a contradiction.

Since $Q$ is quadratic and has a line of zeroes,
it can be decomposed into a product
$$
 Q(v) = f(v) W(v)
$$
where $f\co H\to\R$ is a linear function and $W$ is a linear vector field.
The trajectories of $W$ are contained in level sets of $\phi_0$
and hence bounded. Since $W$ is linear, it follows that either $W=0$ or
these trajectories are ellipses centered at~0.
The latter contradicts our standing assumption that $\phi_0$ is not a Euclidean norm.
Thus $W=0$ and hence $Q=0$.

Thus $R(J(v))(v)=0$ for all $v\in H$.
To finish the proof of the lemma,
observe that if $h\in H^*$ and $v\in\ker h\setminus\{0\}$ then
$h$ is a scalar multiple of $J(v)$.
\end{proof}

Lemma \ref{l:Rhv} implies the following strong property of $\Phi$
(compare with \eqref{e:dPhi(tau)}).

\begin{lemma}\label{l:Rtau}
There exists a vector $\tau\in V\setminus\{0\}$ 
which is tangent to $\Sigma$ at every point of $\ga=H\cap\Sigma$.
\end{lemma}

\begin{proof}
Pick a basis $e_1,e_2,e_3$ of $V$ and decompose $R$ into coordinate components:
$R=R_1e_1+R_2e_2+R_3e_3$ where $R_i\in\Lin(H^*,H^*)$ for $i=1,2,3$.
By Lemma \ref{l:Rhv} we have $\ker R_i(h)\supset\ker h$ and hence
$R_i(h)$ is a scalar multiple of~$h$ for every $h\in H^*$.
Therefore $R_i=\lambda_i\id_{H^*}$ for some $\lambda_i\in\R$.
Equivalently, 
\be\label{e:Ri}
 R_i(h)(v)=\lambda_i h(v)
\ee
for all $h\in H^*$ and $v\in H$.
Define $\tau=\lambda_1e_1+\lambda_2e_2+\lambda_3e_3$.
By \eqref{e:Ri} we have
\be\label{e:Rtau}
 R(h)(v) = h(v)\cdot\tau
\ee
for all $h\in H^*$ and $v\in H$.

Let $v\in\ga$ and pick $h\in H^*$ such that $h(v)\ne 0$.
Substituting \eqref{e:Rtau} into \eqref{e:diffLa}
and dividing by $h(v)$ yields that
$\pi_0(\tau)+1=0$,
therefore $\tau\ne 0$.
Substituting \eqref{e:Rtau} into \eqref{e:dPhiRhv}
and dividing by $h(v)$ yields $d_v\Phi(\tau)=0$,
hence $\tau$ is tangent to $\Sigma$ at~$v$.
\end{proof}

The statement of Lemma \ref{l:Rtau} is equivalent to the property
that there exists a norm non-increasing linear projector from $V$ to $H$.
If this property held for {\em all} planes $H\subset V$,
then the Blaschke-Kakutani characterization of ellipsoids
(\cite{Kak}, see also \cite[\S3.4]{Th} or \cite[\S12.3]{Gruber})
would imply that $\Phi$ is a Euclidean norm.
The next proposition generalizes the Blaschke-Kakutani characterization
to our localized setting.
It is independent of the previous arguments.

\begin{prop}\label{p:kakutani}
Let $(V^3,\Phi)$ be a Banach-Minkowski space
and $\Sigma$ its unit sphere.
Suppose that $\U\subset\Gr_2(V)$ is an open set such that
for every $H\in\U$ there exists a vector $\tau=\tau_H\in V\setminus\{0\}$
which is tangent to $\Sigma$ at every point of $H\cap\Sigma$.
Then $\Phi|_H$ is a Euclidean norm for every $H\in\U$.
\end{prop}

\begin{proof}
Let $v\in H\cap\Sigma$ for some $H\in\U$ and let $v'=-v/\Phi(-v)$
be the opposite $\Phi$-unit vector. Since $\U$ is open,
there exists $H_1\in\U$ such that $v\in H_1$ and $H_1\ne H$.
The strict convexity of $\Phi$ implies that
$\tau_H$ and $\tau_{H_1}$ are linearly independent vectors.
Both tangent planes to $\Sigma$ at $v$ and $v'$ contain
these two vectors, hence these two tangent
planes are parallel.

Let $U_1=\bigcup_{H\in\U} \Gr_1(H)$,
then $U_1$ is an open subset of $\Gr_1(V)$.
Define a map $f\co U_1\to\Gr_2(V)$ as follows.
For a line $\ell\in U$ let $f(\ell)$ be the direction of the
tangent plane to $\Sigma$ at a point $v\in\ell\cap\Sigma$.
By the above argument, this direction is the same for $v$
and the opposite vector $v'$, thus $f$ is well-defined.
The strict convexity of $\Phi$ implies that $f$ is a diffeomorphism
from $U_1$ onto an open subset of $\Gr_2(V)$.

We regard the Grassmannians $\Gr_1(V)$ and $\Gr_2(V)$ as projective planes.
Namely $\Gr_1(V)$ is the projectivization of $V$ and $\Gr_2(V)$ is naturally
identified with the projectivization of~$V^*$:
to each linear function from $V^*\setminus\{0\}$ one associates its kernel.
Each plane $H\in\U$ represents a line in $\Gr_1(V)$, and its $f$-image is
a line in $\Gr_2(V)$ consisting of all planes that contain $\tau_H$.

Thus $f$ is a diffeomorphism between subsets of projective planes,
it is defined on the union of an open set of lines,
and it maps each of these lines to a line.
By the local variant of the fundamental theorem of projective geometry,
these properties imply that $f$ is a restriction of a projective map.
(In fact, it suffices to assume that there are 4 independent families of lines 
that are mapped to lines, see~\cite{Pre}.) Thus there exists a linear map $F\co V\to V^*$
such that $f$ is a restriction of the projectivization of~$F$.

Fix $H\in\U$ and let $\ga=H\cap\Sigma$.
Let $F_H\co H\to H^*$ be the map given by
$F_H(v)=F(v)|_H$ for all $v\in H$.
By construction of $F$ for every $v\in\ga$ 
the line $\ker F_H(v)$ is parallel to the tangent line to $\ga$ at~$v$.
Let $J\co H^*\to H$ be a linear isomorphism which sends every $\alpha\in H^*$
to a vector from its kernel (cf.\ the proof of Lemma \ref{l:Rhv}).
Then the map $W=J\circ F_H\co H\to H$ defines
a nontrivial linear vector field on $H$ and $\ga$ is tangent to this vector field.
As in the proof of Lemma \ref{l:Rhv}, it follows that $\ga$ is an ellipse centered at 0
and hence $\Phi|_H$ is a Euclidean norm.
\end{proof}

Now we finish the proof of Proposition~\ref{p:sections}
and deduce Theorem~\ref{t:mono}.

\begin{proof}[\bf Proof of Proposition \ref{p:sections}]
Let $\U\subset\Gr_2(V)$ be as in Proposition~\ref{p:sections}.
In the set-up preceding Lemma \ref{l:Rtau},
$H$ was an arbitrary plane from $\U$, hence
Lemma \ref{l:Rtau} applies to all planes $H\in\U$.
Thus the assumptions of Proposition~\ref{p:kakutani}
are satisfied and it implies that $\Phi|_H$
is a Euclidean norm for every $H\in\U$.
%This proves Proposition \ref{p:sections}.
\end{proof}

\begin{proof}[\bf Proof of Theorem \ref{t:mono}]
Let $M=(M^2,\phi)$ be a monochromatic Finsler manifold,
$V=(V^3,\Phi)$ a Banach-Minkowski space, and $f\co M\to V$ an isometric embedding.
Let $G\co M\to\Gr_2(V)$ be the Gauss map of $f$
defined by $G(x)=\operatorname{Im} d_xf$ for $x\in M$.
If the second fundamental form of $f$ is non-degenerate at $p\in M$,
then the derivative of $G$ at $p$ is non-degenerate and
hence the image of $G$ contains a neighborhood of $G(p)$ in $\Gr_2(V)$.
Since $\phi$ is monochromatic, this implies that that the restrictions of $\Phi$
to all planes from this neighborhood are isometric.
By Proposition \ref{p:sections} this implies that
the norm $\phi|_{T_pM}$ is Euclidean and hence $M$ is Riemannian.
%Thus  Theorem \ref{t:mono} follows from Proposition \ref{p:sections}.
\end{proof}

\begin{remark}\label{rem:converse}
Conversely, Theorem \ref{t:mono} easily implies Proposition \ref{p:sections}.
Indeed, if $(V,\Phi)$ and $\U$ are as in Proposition \ref{p:sections}
then for any surface $M\subset V$ whose tangent planes all belong to $\U$,
the induced Finsler metric on $M$ is monochromatic.
Suppose that Proposition \ref{p:sections} fails and let
$x,y,z$ be coordinates in $V$ such that $\U$ contains the plane $\{z=0\}$.
Then a small neighborhood of $0$ in the surface $\{z=x^2+y^2\}$
is a counter-example to Theorem~\ref{t:mono}.
\end{remark}

\section{Norms with many Euclidean sections}
\label{sec:riem}

In this section we finish the proof of Theorem \ref{t:sections}'
and consider isometric embeddings of Riemannian
2-manifolds into 3-dimensional Banach-Minkowski spaces.
It turns out that they are essentially no different
from isometric embeddings into Euclidean spaces,
see Proposition \ref{p:riemannian}.
Both results are based on the following lemma
which localizes the well-known fact that
a normed space is Euclidean if all its 2-dimensional
subspaces are.
%See also \cite{Petty}, \cite{AlMa}, \cite{BiGr}, \cite[Chapter 7]{Gardner}
%for other (global) ellipsoid characterizations in terms of cross-sections.

\begin{lemma}\label{l:quadratic}
Let $V=(V^3,\Phi)$ be a Banach-Minkwoski space
and $\Gamma\subset\Gr_2(V)$ be a set of planes
such that the set $\bigcup\Gamma:=\bigcup_{H\in\Gamma}H\subset V$
has a nonempty interior.
Suppose that for every $H\in\Gamma$ the norm $\Phi|_H$ is Euclidean.
Then there exists a Euclidean norm on $V$ whose restriction
to every plane from $\Gamma$ coincides with the restriction of $\Phi$.
\end{lemma}

\begin{remark}
Without the assumption that $\bigcup\Gamma$ has a nonempty interior
a weaker conclusion holds: There exists a quadratic form $Q$ on $V$
(not necessarily positive definite)
such that $Q|_H=\Phi^2|_H$ for every plane $H\in\Gamma$. 
This can be seen from the proof of Lemma \ref{l:quadratic}.
\end{remark}

\begin{proof}[Proof of Lemma \ref{l:quadratic}]
Consider two cases.

\textit{Case 1: 
There is a line $\ell\in\Gr_1(V)$ contained in at least 3 planes from $\Gamma$.}
Fix such a line $\ell$ and
divide $\Gamma$ into subsets $\Gamma_0$ and $\Gamma_1$
where the planes from $\Gamma_0$ contain~$\ell$
and those from $\Gamma_1$ do not.
Pick a vector $v\in\ell\setminus\{0\}$
and define a quadratic form $Q$ on $V$ by $Q(x)=\frac12 d^2_v(\Phi^2)(x,x)$
for all $x\in V$.
%Here we use the smoothness of $\Phi$ on $V\setminus\{0\}$.
Since $\Phi^2$ is strictly convex, $Q$ is positive definite
and hence it is a square of some Euclidean norm.
If $\Phi$ itself is Euclidean then $Q=\Phi^2$.
This observation applied to the restriction of $\Phi$ to
a plane $H\in\Gamma_0$ implies that $Q|_H=\Phi^2|_H$ for every such plane.

Now consider a plane $H\in\Gamma_1$. The restriction of $\Phi^2$ to $H$
is a quadratic form which equals $Q$ on at least 3 lines through the origin
(these lines are the intersections of $H$ with the planes from $\Gamma_0$).
Since a quadratic form on the plane is uniquely determined by its values
at three pairwise linearly independent vectors, it follows that
$\Phi^2|_H=Q|_H$.
This proves the lemma under the assumption of Case~1.

\smallbreak
\textit{Case 2: No three planes from $\Gamma$ have a common line.}
Pick $H_0\in\Gamma$ and a vector $v_0\in V\setminus H_0$
such that $\Phi(v_0)\le 1$ and $\Phi(-v_0)\le 1$.
Consider the affine plane $H_1=H_0+v_0$.
Every plane from $\Gamma$ except $H_0$ intersects $H_1$ by an affine line;
let $\Lambda$ denote the set of all such lines.
Since $\Phi$ is a Euclidean norm on every plane from~$\Gamma$,
$\Phi^2$ is a quadratic polynomial on every line from $\Lambda$.
By the assumption of Case~2, no two lines from $\Lambda$ are parallel
and no three of them have a common point.

Define a function $F\co H_1\to\R$ by
$$
 F(v) = \Phi^2(v)-\Phi^2(v-v_0) , \qquad v\in H_1.
$$
The restriction of $F$ to any line from $\Lambda$ is
the difference of two quadratic polynomials
and hence a polynomial of degree at most~2.
In fact, its degree is no greater than~1.
Indeed, the triangle inequality for $\Phi$ implies that for every $v\in H_1$,
%$$
% \Phi(v-v_0)\le \Phi(v)+\Phi(-v_0) \le \Phi(v)+1
%$$
%and
%$$
% \Phi(v)\le \Phi(v-v_0)+\Phi(v_0) \le \Phi(v-v_0)+1 ,
%$$
%hence 
$$
 |\Phi(v)-\Phi(v-v_0)| \le \max\{\Phi(v_0),\Phi(-v_0)\} \le 1 ,
$$
and therefore
$$
 |F(v)| = |\Phi(v)-\Phi(v-v_0)|\cdot |\Phi(v)+\Phi(v-v_0)| \le 2\Phi(v)+1 .
$$
This implies that $|F|$ has at most linear growth at infinity,
hence its restriction to a line cannot be a degree 2 polynomial.

Thus the restriction of $F$ to any line from $\Lambda$ is an affine function.
Pick two lines $\ell_1,\ell_2\in\Lambda$. Since they are not parallel, there
exist an affine function $\tilde F\co H_1\to\R$ such that
$\tilde F|_{\ell_1\cup\ell_2}=F|_{\ell_1\cup\ell_2}$.
Every line $\ell\in\Lambda\setminus\{\ell_1,\ell_2\}$ intersects $\ell_1\cup\ell_2$
at two distinct points. An affine function on $\ell$ is uniquely determined
by its values at these two points, hence $\tilde F|_\ell=F|_\ell$.
Since $\Phi^2|_{H_0}$ is a quadratic form and $\tilde F$ is affine,
there is a quadratic form $Q\co V\to\R$ such that $Q|_{H_0}=\Phi^2|_{H_0}$
and $Q(v) = Q(v-v_0)+\tilde F(v)$ for all $v\in H_1$.
If $v\in\ell\in\Lambda$ then
$Q(v) = \Phi^2(v-v_0)+F(v) = \Phi^2(v)$
since $\tilde F|_\ell=F|_\ell$ as shown above.
By homogeneity it follows that $Q=\Phi^2$ on $\bigcup\Gamma$.

It remains to prove that $Q$ is positive definite.
Suppose the contrary and choose
$v,w\in V\setminus\{0\}$ such that $v$ is an interior point of $\bigcup\Gamma$
and $Q(w)\le 0$.
Then 
$$
d^2_v(\Phi^2)(w,w)=(d^2_vQ)(w,w)=2Q(w)\le 0,
$$
contrary to the strict convexity of $\Phi^2$.
This finishes the proof of Lemma \ref{l:quadratic}.
\end{proof}

\begin{proof}[\bf Proof of Theorem \ref{t:sections}']
Let $(V,\Phi)$ and $\U$ be as in Theorem \ref{t:sections}',
and let $n=\dim V$.
First consider the case $n=3$.
By Proposition \ref{p:sections} the norm $\Phi|_{H}$
is Euclidean for every $H\in\U$.
Since $\U$ is an open set of planes, $\bigcup\U\setminus\{0\}$ 
is an open set in $V$.
Hence Lemma~\ref{l:quadratic} applied to $\Gamma=\U$
implies the last claim of the theorem.
This proves the theorem in dimension~3.

If $n>3$, one can apply this to any
3-dimensional subspace of $V$ containing at least one plane from $\U$.
This implies that $\Phi|_{H}$
is a Euclidean norm for every $H\in\U$. It remains to prove that there is
a Euclidean norm on $V$ extending all these two-dimensional norms.
Fix $v\in\bigcup\U\setminus\{0\}$ and apply the 3-dimensional case %of the theorem
to all 3-dimensional subspaces containing $v$ and a fixed plane $H\in\U$ containing~$v$.
This yields a Euclidean norm on each of these subspaces.
Similarly
to the proof of Lemma \ref{l:quadratic}, the squares of these norms
are restrictions of the quadratic form $Q_v$ defined by
$Q_v(x)=\frac12d^2_v(\Phi^2)(x,x)$, $x\in V$.
Thus $\Phi^2=Q_v$ in a neighborhood of $v$.
A quadratic form $Q_v$ with this property obviously cannot change
if $v$ varies continuously. Since $\U$ is connected, it follows
that $Q_v$ is the same quadratic form for all $v\in\bigcup\U\setminus\{0\}$.
The square root of this quadratic form is the desired Euclidean norm on~$V$.
\end{proof}

Now we return to isometric embeddings.

\begin{prop}\label{p:riemannian}
Let $M$ be a connected Riemannian 2-manifold, $V=(V^3,\Phi)$
a Banach-Minkowski space, and $f\co M\to V$ an isometric immersion.
Then there exists a Euclidean norm on $V$ such that $f$ is isometric
with respect to this norm.
\end{prop}

\begin{proof}
Define $\Gamma\subset\Gr_2(V)$ by $\Gamma=\{\operatorname{Im} d_p f : p\in M\}$.
For every plane $H\in\Gamma$ the restriction $\Phi|_H$ is a Euclidean norm,
and we need to prove that there is a Euclidean norm on $V$ extending all these
two-dimensional norms.

If the second fundamental form of $f$ vanishes everywhere, then $\Gamma$
consists of one plane and the result is trivial.
If the second fundamental form at some point $p\in M$ and some vector $v\in T_pM$
is nonzero, then $d_pf(v)$ is an interior point of $\bigcup\Gamma$. In this case the
result follows from Lemma \ref{l:quadratic}.
\end{proof}

The next corollary asserts that the answer to Question \ref{main-question} is
affirmative for all monochromatic metrics.

\begin{cor}\label{cor:mono}
Let $M^2$ be a monochromatic (possibly Riemannian)
Finsler manifold, $p\in M$, and
$f_i\co M\to V_i$, $i=1,2$, smooth isometric embeddings
where $V_1$ and $V_2$ are 3-dimensional Banach-Minkowski spaces.
Then the second fundamental forms of $f_i$ at $p$
are of the same type: either both are definite, or both are semi-definite,
or both are degenerate.
\end{cor}

\begin{proof}
If the metric of $M$ is not Riemannian then by Theorem \ref{t:mono}
both second fundamental forms are degenerate.
If the metric is Riemannian then by Proposition \ref{p:riemannian}
the norms of $V_1$ and $V_2$ can be replaced by Euclidean ones.
Then by Gauss' theorem the type of the second fundamental form
is determined by the sign of the Riemannian curvature at~$p$.
\end{proof}

\section{A non-embeddable example}
\label{sec:example}

The following example describes a class of Finsler metrics on $\R^2$
that do not admit local isometric embeddings to 3-dimensional
Banach-Minkowski spaces. The proof of non-embeddability is
given in Proposition \ref{p:nonembed}.

\begin{example}\label{ex:nonembed}
Let $\phi_0$ be a Banach-Minkowski norm on $\R^2$
such that it is not Euclidean and moreover there in no open set of $\R^2$
where the restriction of $\phi_0$ equals the restriction of a Euclidean norm.

For $\theta\in\R$, let $R^\theta$ denote the rotation of $\R^2$ by angle $\theta$:
$$
 R^\theta(\xi,\eta)=(\cos\theta\cdot\xi-\sin\theta\cdot\eta,\sin\theta\cdot\xi+\cos\theta\cdot\eta) .
$$
Define a Finsler metric $\phi$ on $\R^2$ by
$$
 \phi(x,y,\xi,\eta) = \phi_0(R^y(\xi,\eta))
$$
where $(x,y)$ and $(\xi,\eta)$ are coordinates of a point in $\R^2$
and a tangent vector in $T_{(x,y)}\R^2\cong\R^2$, resp.
\end{example}

Obviously norms $\phi_0$ satisfying the above requirements do exist.
For an explicit example, one can take
$\phi_0(\xi,\eta) = \sqrt{\xi^2+\eta^2} + \sqrt{\xi^2+2\eta^2}$
(or, in fact, any analytic formula that defines a non-Euclidean norm).

\begin{prop}\label{p:nonembed}
Let $\phi$ be as in Example \ref{ex:nonembed}.
Then no open subset of $(\R^2,\phi)$
admits a smooth isometric embedding to a 3-dimensional Banach-Minkowski space.
\end{prop}

\begin{proof}
Arguing by contradiction, 
let $U\subset\R^2$ be an open set, $\Phi$ is a Banach-Minkowski norm
in $\R^3$ and assume that $f\co(U,\phi)\to(\R^3,\Phi)$ 
is a smooth isometric embedding. Define $M=f(U)$.

By construction $\phi$ is monochromatic.
Hence by Theorem \ref{t:mono} the second fundamental form of $f$ is degenerate everywhere.
Therefore $f$, regarded as a surface in the Euclidean $\R^3$,
is a developable surface.
It is well-known
that every developable surface in $\R^3$ is a ruled surface,
i.e., it is the union of straight line segments, see e.g.\ \cite[\S5-8]{doCa}.
Furthermore the tangent planes to the surface are constant along every such segment.

Consider one of such segments in $M$ parametrized with a constant speed.
It can be written as $f(\ga(t))$ where $\ga=\ga(t)$, $t\in(-\ep,\ep)$, is a smooth curve in $U$.
Since the segment is a length-minimizing curve in $(\R^3,\Phi)$,
$\ga$ is a Finsler geodesic in $(U,\phi)$.
Let $\ga(t)=(x(t),y(t))$.
Since $\phi_0$ is not a Euclidean norm, its group of self-isometries is discrete.
Hence there is a unique continuous family of linear maps 
$I_t\co T_{\ga(0)} U\to T_{\ga(t)} U$ such that $I_0$ is the identity and
each $I_t$ is an isometry between the Finsler norms at $\ga(0)$ and $\ga(t)$.
In the standard coordinates this family is given by $I_t=R^{-y(t)}$.

Recall that the tangent plane $T_{f(\ga(t))}M=\operatorname{Im} d_{\ga(t)}f$ is constant
along the segment.
Due to their uniqueness,
the isometries $I_t$ correspond to the identity map of this plane, i.e.
\be\label{e:dfIt}
d_{\ga(t)}f\circ I_t = d_{\ga(0)} f
\ee
for all $t$.
Since $f\circ\ga$ is a constant-speed straight line segment,
its velocity vector $\frac d{dt}f(\ga(t))=d_{\ga(t)}f(\dot\ga(t))$
does not depend on~$t$. This and \eqref{e:dfIt} imply that
$$
  d_{\ga(t)}f(\dot\ga(t)) = d_{\ga(0)}f(\dot\ga(0)) = d_{\ga(t)}f(I_t(\dot\ga(0)))
$$
%=\frac d{dt}f(\ga(t))=d_{\ga(0)}f(\dot\ga(0))$.
and therefore 
\be\label{e:dotga}
\dot\ga(t) = I_t(\dot\ga(0)) = R^{-y(t)}(\dot\ga(0)). 
\ee

The Euler-Lagrange equations for the geodesic $\ga$ are
\be\label{e:euler}
\begin{cases}
  \frac d{dt} \frac{\pd\phi}{\pd\xi} (\ga(t),\dot\ga(t)) = \frac{\pd\phi}{\pd x} (\ga(t),\dot\ga(t)) , \\
  \frac d{dt} \frac{\pd\phi}{\pd\eta} (\ga(t),\dot\ga(t)) = \frac{\pd\phi}{\pd y} (\ga(t),\dot\ga(t)) .
\end{cases}
\ee
where $\phi$ is regarded as a function of variables $x,y,\xi,\eta$ as in Example \ref{ex:nonembed}.
Since $\phi$ is preserved by translations along the $x$-axis, the first equation in \eqref{e:euler}
imply that $\frac{\pd\phi}{\pd\xi} (\ga(t),\dot\ga(t))$ is constant.
(One can also see this constant as the Noether integral associated 
with the group of horizontal translations.)
Observe that
$$
\frac{\pd\phi}{\pd\xi} (\ga(t),\dot\ga(t))
= d_{\dot\ga(t)}(\phi_0\circ R^{y(t)}) (e_1)
= (d_{\dot\ga(0)}\phi_0) (R^{y(t)}(e_1))
$$
due to \eqref{e:dotga}, where $e_1$ is the first coordinate vector of $\R^2$.
Since $d_{\dot\ga(0)}\phi_0\co\R^2\to\R$ is a nontrivial linear map
and $R^{y(t)}(e_1)$ belongs to the Euclidean unit circle, it follows
that $R^{y(t)}(e_1)$ is constant. Hence $y(t)$ is constant.

Thus $\ga$ is a (constant-speed parametrized) horizontal segment in $U\subset\R^2$.
Since $M$ contains a segment through every point, it follows that
all horizontal segments in $U$ are geodesics of~$\phi$.
Then the second equation in \eqref{e:euler} implies that
$$
 \frac{\pd\phi}{\pd y}(x,y,0,1) = 0
$$
for all $(x,y)\in U$.
Since $\phi$ does not depend on $x$, it follows that 
the value 
$$
\phi(x,y,0,1) = \phi_0(R^y(e_1)), \qquad (x,y)\in U,
$$
is constant. Thus $\phi_0$ is constant on an arc of the Euclidean unit circle
and hence it is proportional to the standard Euclidean norm on an open subset of~$\R^2$.
This contradicts the requirements of Example \ref{ex:nonembed}.
\end{proof}


\begin{thebibliography}{99}

%\bibitem{AlMa}
%J. Alonso, P. Mart\'in, 
%{\em Some characterizations of ellipsoids by sections},
%Discrete and Computational Geometry {\bf 31}, 2004, 643--654.

\bibitem{AMU}
H. Auerbach, S. Mazur, S. Ulam,
{\em Sur une propri\'et\'e caract\'eristique de l'ellipso\"ide},
Monatsh. Math. Phys. {\bf 42} (1935), no. 1, 45--48.

\bibitem{Banach32}
S. Banach, {\em Th\'eorie des op\'erations lin\'eaires}, Monografie Matematyczne, vol. 1, Warszawa, 1932.

\bibitem{Bao} 
D. Bao,
{\em On two curvature-driven problems in Riemann-Finsler geometry},
In ``Finsler geometry, Sapporo 2005---in memory of Makoto Matsumoto'',
Adv. Stud. Pure Math., {\bf 48}, Math. Soc. Japan, Tokyo, 2007, pp. 19--71.

%\bibitem{BiGr}
%G. Bianchi, P. M. Gruber,
%{\em Characterizations of ellipsoids},
%Arch. Math. {\bf 49} (1987), 344--350.

\bibitem{BI93}
D. Burago, S. Ivanov,
{\em Isometric embeddings of Finsler manifolds},
Algebra i Analiz {\bf 5} (1993), 179--192 (Russian);
English translation: St.~Petersburg Math. J., {\bf 5} (1994), 159--169.

\bibitem{BI11}
D. Burago, S. Ivanov,
{\em On intrinsic geometry of surfaces in normed spaces},
Geom. Topol. {\bf 15} (2011), no. 4, 2275--2298.
%{\tt arXiv:1011.5670}.

\bibitem{doCa}
M. P. do Carmo, {\em Differential geometry of curves and surfaces},
Prentice-Hall, 1976.

\bibitem{Gardner}
R. J. Gardner, {\em Geometric Tomography}, 2nd edition,
Encyclopedia of mathematics and its applications, vol. 58,
Cambridge University Press, 2006.

\bibitem{Gro67}
M. Gromov, {\em On a geometric conjecture of Banach}, 
Izv. Akad. Nauk SSSR Ser. Mat. {\bf 31} (1967), no. 5, 1105--1114  (Russian);
English translation: Mathematics of the USSR-Izvestiya {\bf 1} (1967), no. 5, 1055--1064.

\bibitem{Gruber}
P. M. Gruber, {\em Convex and Discrete Geometry}, 
Grundlehren der mathematischen Wissenschaften, vol. 336,
Springer, 2007.

\bibitem{Kak}
S. Kakutani, {\em Some characterizations of inner product spaces}, Jap. J. Math. {\bf 16} (1939), 93--97.

%\bibitem{Kas}
%E. Kasner,{\em The characterization of collineations}, Bull. Amer. Math. Soc. {\bf 9} (1903), no. 10, 545--546.

%\bibitem{Petty}
%C. M. Petty, {\em Ellipsoids}, In ``Convexity and Its Applications'' (P. M. Gruber and J. M. Wills, eds.),
%Springer, 1983, pp. 264--276.

\bibitem{Pre} 
W. Prenowitz,
{\em The characterization of plane collineations in terms of homologous families of lines},
Trans. Amer. Math. Soc. {\bf 38} (1935), 564--599.

\bibitem{Th}
A. C. Thompson, {\em Minkowski Geometry}, 
Encyclopedia of mathematics and its applications, vol. 63,
Cambridge University Press, 1996.

\bibitem{Shen}
Z. Shen, {\em On Finsler geometry of submanifolds},
Math. Ann {\bf 311} (1998), 549--576.

\end{thebibliography}
\end{document}